\documentclass[12pt]{amsart}
\usepackage{amssymb}
\usepackage[all]{xy}

\newtheorem{theorem}{Theorem}[section]
\newtheorem{definition}[theorem]{Definition}
\newtheorem{proposition}[theorem]{Proposition}

\begin{document}

\title{Tannakian duality for affine homogeneous spaces}

\author{Teodor Banica}
\address{T.B.: Department of Mathematics, University of Cergy-Pontoise, F-95000 Cergy-Pontoise, France. {\tt teo.banica@gmail.com}}


\begin{abstract}
Associated to any closed quantum subgroup $G\subset U_N^+$ and any index set $I\subset\{1,\ldots,N\}$ is a certain homogeneous space $X_{G,I}\subset S^{N-1}_{\mathbb C,+}$, called affine homogeneous space. We discuss here the abstract axiomatization of the algebraic manifolds $X\subset S^{N-1}_{\mathbb C,+}$ which can appear in this way, by using Tannakian duality methods.
\end{abstract}

\maketitle

\section*{Introduction}

The compact quantum groups were introduced by Woronowicz in \cite{wo1}, \cite{wo2}. These are abstract objects, having no points in general, generalizing at the same time the usual compact groups, and the duals of the discrete groups. The compact quantum groups have no Lie algebra in general, but analogues of the Peter-Weyl theory, of Tannakian duality, and of the Weingarten integration formula, are available. Thus, we have here some interesting examples of noncommutative manifolds, which are definitely algebraic, and which are probably a bit Riemannian too, because we can integrate on them.

The present work is a continuation of \cite{ban}, which was concerned with the integration theory over the associated homogeneous spaces. The main finding there was the fact that, in order to have a good integration theory, one must restrict the attention to a certain special class of homogeneous spaces, called ``affine''. To be more precise, associated to any closed subgroup $G\subset U_N^+$ and any index set $I\subset\{1,\ldots,N\}$ is a certain homogeneous space $X_{G,I}\subset S^{N-1}_{\mathbb C,+}$, called affine. In the classical case this space appears as $X_{G,I}=G/(G\cap C_N^I)$, where $C_N^I\subset U_N$ is the group of unitaries fixing the vector $\xi_I=\frac{1}{\sqrt{|I|}}\sum_{i\in I}e_i$. In general, however, there are many new twists and questions coming from noncommutativity. Importantly, this construction covers many interesting examples. See \cite{ban}.

One question left open in \cite{ban} was that of finding an abstract axiomatization of the algebraic manifolds $X\subset S^{N-1}_{\mathbb C,+}$ which can appear in this way. We will answer here this question, with a Tannakian characterization of such manifolds. We believe that some further improvements of this result can lead to an axiomatization of the ``easy algebraic manifolds'', which was the main question in \cite{ban}, and which remains open. 

The paper is organized as follows: 1 is a preliminary section, in 2-3 we state and prove the Tannakian duality result, and 4 contains a number of further results.

\section{Homogeneous spaces}

We use Woronowicz's quantum group formalism in \cite{wo1}, \cite{wo2}, with the extra axiom $S^2=id$. To be more precise, the definition that we will need is:

\begin{definition}
Assume that $A$ is a unital $C^*$-algebra, and that $u\in M_N(A)$ is a unitary matrix whose coefficients generate $A$, such that the following formulae define morphisms of $C^*$-algebras $\Delta:A\to A\otimes A,\varepsilon:A\to\mathbb C$ and $S:A\to A^{opp}$:
$$\Delta(u_{ij})=\sum_ku_{ik}\otimes u_{kj}\quad,\quad\varepsilon(u_{ij})=\delta_{ij}\quad,\quad S(u_{ij})=u_{ji}^*$$
We write then $A=C(G)$, and call $G$ a compact matrix quantum group.
\end{definition}

The above maps $\Delta,\varepsilon,S$ are called comultiplication, counit and antipode. The basic examples include the compact Lie groups $G\subset U_N$, their $q$-deformations at $q=-1$, and the duals of the finitely generated discrete groups $\Gamma=<g_1,\ldots,g_N>$. See \cite{wo1}.

We recall that the free unitary quantum group $U_N^+$, constructed by Wang in \cite{wan}, and the corresponding free complex sphere $S^{N-1}_{\mathbb C,+}$, from \cite{bgo}, are constructed as follows:
\begin{eqnarray*}
C(U_N^+)&=&C^*\left((u_{ij})_{i,j=1,\ldots,N}\Big|u^*=u^{-1},u^t=\bar{u}^{-1}\right)\\
C(S^{N-1}_{\mathbb C,+})&=&C^*\left(x_1,\ldots,x_N\Big|\sum_ix_ix_i^*=\sum_ix_i^*x_i=1\right)
\end{eqnarray*}

Here both the algebras on the right are by definition universal $C^*$-algebras.

Following \cite{ban}, we can now formulate the following definition:

\begin{definition}
An affine homogeneous space over a closed subgroup $G\subset U_N^+$ is a closed subset $X\subset S^{N-1}_{\mathbb C,+}$, such that there exists an index set $I\subset\{1,\ldots,N\}$ such that
$$\alpha(x_i)=\frac{1}{\sqrt{|I|}}\sum_{j\in I}u_{ij}\quad,\quad
\Phi(x_i)=\sum_ju_{ij}\otimes x_j$$
define morphisms of $C^*$-algebras, satisfying $(\int_G\otimes id)\Phi=\int_G\alpha(.)1$. 
\end{definition}

Here, and in what follows, a closed subspace $Y\subset Z$ corresponds by definition to a quotient map $C(Z)\to C(Y)$. As for $\int_G$, this is the Haar integration. See \cite{wo1}.

As a first observation, the coaction condition $(id\otimes\Phi)\Phi=(\Delta\otimes id)\Phi$ is satisfied, and we have as well $(id\otimes\alpha)\Phi=\Delta\alpha$. In the case where $\alpha$ is injective, we have:

\begin{proposition}
When $\alpha$ is injective we must have $X=X_{G,I}^{min}$, where:
$$C(X_{G,I}^{min})=\left<\frac{1}{\sqrt{|I|}}\sum_{j\in I}u_{ij}\Big|i=1,\ldots,N\right>\subset C(G)$$
Moreover, $X_{G,I}^{min}$ is affine homogeneous, for any $G\subset U_N^+$, and any $I\subset\{1,\ldots,N\}$.
\end{proposition}

\begin{proof}
The first assertion is clear from definitions. Regarding now the second assertion, consider the variables $z_i=\frac{1}{\sqrt{|I|}}\sum_{j\in I}u_{ij}\in C(G)$. We have then:
$$\Delta(z_i)=\frac{1}{\sqrt{|I|}}\sum_{j\in I}\sum_ku_{ik}\otimes u_{kj}=\sum_ku_{ik}\otimes z_k$$

Thus we have a coaction map as in Definition 1.2, given by $\Phi=\Delta$, and the ergodicity condition, namely $(\int_G\otimes id)\Delta=\int_G(.)1$, holds as well, by definition of $\int_G$. See \cite{ban}.
\end{proof}

Given exponents $e_1,\ldots,e_k\in\{1,*\}$, consider the following quantities:
$$P_{i_1\ldots i_k,j_1\ldots j_k}=\int_Gu_{i_1j_1}^{e_1}\ldots u_{i_kj_k}^{e_k}$$

Once again following \cite{ban}, we have the following result:

\begin{proposition}
We must have $X\subset X_{G,I}^{max}$, as subsets of $S^{N-1}_{\mathbb C,+}$, where:
$$C(X_{G,I}^{max})=C(S^{N-1}_{\mathbb C,+})\Big/\left<(Px^{\otimes k})_{i_1\ldots i_k}=\frac{1}{\sqrt{|I|^k}}\sum_{j_1\ldots j_k\in I}P_{i_1\ldots i_k,j_1\ldots j_k}\big|\forall k,\forall i_1,\ldots i_k\right>$$
Moreover, $X_{G,I}^{max}$ is affine homogeneous, for any $G\subset U_N^+$, and any $I\subset\{1,\ldots,N\}$.
\end{proposition}

\begin{proof}
The idea here is that the ergodicity condition $(\int_G\otimes id)\Phi=\int_G\alpha(.)1$ produces the relations in the statement. To be more precise, observe that we have:
\begin{eqnarray*}
&&\left(\int_G\otimes id\right)\Phi=\int\alpha(.)1\\
&\iff&\left(\int_G\otimes id\right)\Phi(x_{i_1}^{e_1}\ldots x_{i_k}^{e_k})=\frac{1}{\sqrt{|I|^k}}\int_G\alpha(x_{i_1}^{e_1}\ldots x_{i_k}^{e_k}),\forall k,\forall i_1,\ldots i_k\\
&\iff&\sum_{j_1\ldots j_k}P_{i_1\ldots i_k,j_1\ldots j_k}x_{j_1}^{e_1}\ldots x_{j_k}^{e_k}=\frac{1}{\sqrt{|I|^k}}\sum_{j_1\ldots j_k\in I}P_{i_1\ldots i_k,j_1\ldots j_k},\forall k,\forall i_1,\ldots i_k
\end{eqnarray*}

Thus we have $X\subset X_{G,I}^{max}$, and the last assertion is standard as well. See \cite{ban}.
\end{proof}

We will need one more general result from \cite{ban}, namely an extension of the Weingarten integration formula \cite{bco}, \cite{cma}, \cite{csn}, \cite{wei}, to the affine homogeneous space setting:

\begin{proposition}
Assuming that $G\to X$ is an affine homogeneous space, with index set $I\subset\{1,\ldots,N\}$, the Haar integration functional $\int_X=\int_G\alpha$ is given by
$$\int_Xx_{i_1}^{e_1}\ldots x_{i_k}^{e_k}=\sum_{\pi,\sigma\in D}(\xi_\pi)_{i_1\ldots i_k}K_I(\sigma)W_{kN}(\pi,\sigma)$$
where $\{\xi_\pi|\pi\in D\}$ is a basis of $Fix(u^{\otimes k})$, $W_{kN}=G_{kN}^{-1}$ with $G_{kN}(\pi,\sigma)=<\xi_\pi,\xi_\sigma>$ is the associated Weingarten matrix, and $K_I(\sigma)=\frac{1}{\sqrt{|I|^k}}\sum_{b_1\ldots b_k\in I}\overline{(\xi_\sigma)_{b_1\ldots b_k}}$.
\end{proposition}

\begin{proof}
By using the Weingarten formula for the quantum group $G$, we have:
\begin{eqnarray*}
\int_Xx_{i_1}^{e_1}\ldots x_{i_k}^{e_k}
&=&\frac{1}{\sqrt{|I|^k}}\sum_{b_1\ldots b_k\in I}\int_Gu_{i_1b_1}^{e_1}\ldots u_{i_kb_k}^{e_k}\\
&=&\frac{1}{\sqrt{|I|^k}}\sum_{b_1\ldots b_k\in I}\sum_{\pi,\sigma\in D}(\xi_\pi)_{i_1\ldots i_k}\overline{(\xi_\sigma)_{b_1\ldots b_k}}W_{kN}(\pi,\sigma)
\end{eqnarray*}

But this gives the formula in the statement, and we are done. See \cite{ban}.
\end{proof}

Finally, here is a natural example of an intermediate space $X_{G,I}^{min}\subset X\subset X_{G,I}^{max}$:

\begin{proposition}
Given a closed quantum subgroup $G\subset U_N^+$, and a set $I\subset\{1,\ldots,N\}$, if we consider the following quotient algebra
$$C(X_{G,I}^{med})=C(S^{N-1}_{\mathbb C,+})\Big/\left<\sum_{a_1\ldots a_k}\xi_{a_1\ldots a_k}x_{a_1}^{e_1}\ldots x_{a_k}^{e_k}=\frac{1}{\sqrt{|I|^k}}\sum_{b_1\ldots b_k\in I}\xi_{b_1\ldots b_k}\Big|\forall k,\forall\xi\in Fix(u^{\otimes k})\right>$$
we obtain in this way an affine homogeneous space $G\to X_{G,I}$.
\end{proposition}

\begin{proof}
We know from Proposition 1.4 above that $X_{G,I}^{max}\subset S^{N-1}_{\mathbb C,+}$ is constructed by imposing to the standard coordinates the conditions $Px^{\otimes k}=P^I$, where:
$$P_{i_1\ldots i_k,j_1\ldots j_k}=\int_Gu_{i_1j_1}^{e_1}\ldots u_{i_kj_k}^{e_k}\quad,\quad
P^I_{i_1\ldots i_k}=\frac{1}{\sqrt{|I|^k}}\sum_{j_1\ldots j_k\in I}P_{i_1\ldots i_k,j_1\ldots j_k}$$

According to the Weingarten integration formula for $G$, we have:
\begin{eqnarray*}
(Px^{\otimes k})_{i_1\ldots i_k}&=&\sum_{a_1\ldots a_k}\sum_{\pi,\sigma\in D}(\xi_\pi)_{i_1\ldots i_k}\overline{(\xi_\sigma)_{a_1\ldots a_k}}W_{kN}(\pi,\sigma)x_{a_1}^{e_1}\ldots x_{a_k}^{e_k}\\
P^I_{i_1\ldots i_k}&=&\frac{1}{\sqrt{|I|^k}}\sum_{b_1\ldots b_k\in I}\sum_{\pi,\sigma\in D}(\xi_\pi)_{i_1\ldots i_k}\overline{(\xi_\sigma)_{b_1\ldots b_k}}W_{kN}(\pi,\sigma)
\end{eqnarray*}

Thus $X_{G,I}^{med}\subset X_{G,I}^{max}$, and the other assertions are standard as well. See \cite{ban}.
\end{proof}

We can now put everything together, as follows:

\begin{theorem}
Given a closed subgroup $G\subset U_N^+$, and a subset $I\subset\{1,\ldots,N\}$, the affine homogeneous spaces over $G$, with index set $I$, have the following properties:
\begin{enumerate}
\item These are exactly the intermediate subspaces $X_{G,I}^{min}\subset X\subset X_{G,I}^{max}$ on which $G$ acts affinely, with the action being ergodic.

\item For the minimal and maximal spaces $X_{G,I}^{min}$ and $X_{G,I}^{max}$, as well as for the intermediate space $X_{G,I}^{med}$ constructed above, these conditions are satisfied.

\item By performing the GNS construction with respect to the Haar integration functional $\int_X=\int_G\alpha$ we obtain the minimal space $X_{G,I}^{min}$.
\end{enumerate}
We agree to identify all these spaces, via the GNS construction, and denote them $X_{G,I}$.
\end{theorem}

\begin{proof}
This follows indeed by combining the various results and observations formulated above. Once again, for full details on all these facts, we refer to \cite{ban}.
\end{proof}

Observe the similarity with what happens for the $C^*$-algebras of the discrete groups, where the various intermediate algebras $C^*(\Gamma)\to A\to C_{red}^*(\Gamma)$ must be identified as well, in order to reach to a unique noncommutative space $\widehat{\Gamma}$. For details here, see \cite{wo1}.

Regarding now the basic examples of such spaces, we have:

\begin{proposition}
Given $N\in\mathbb N$ and $I\subset\{1,\ldots,N\}$, the following hold:
\begin{enumerate}
\item In the classical case, $G\subset U_N$, we have $X_{G,I}=G/(G\cap C_N^I)$, where $C_N^I\subset U_N$ is the group of unitaries fixing the vector $\xi_I=\frac{1}{\sqrt{|I|}}(\delta_{i\in I})_i$.

\item In the group dual case, $G=\widehat{\Gamma}\subset U_N^+$  with $\Gamma=<g_1,\ldots,g_N>$, embedded via $u_{ij}=\delta_{ij}g_i$, we have $X_{G,I}=\widehat{\Gamma}_I$, with $\Gamma_I=<g_i|i\in I>\subset\Gamma$.
\end{enumerate}
\end{proposition}

\begin{proof}
In this statement (1) follows from the fact that the action $G\curvearrowright X_{G,I}$ can be shown to be transitive, and the stabilizer of $\xi_I$ is the group $G\cap C_N^I$ in the statement. As for (2), this follows directly from Definition 1.2, by using $u_{ij}=\delta_{ij}g_i$. See \cite{ban}.
\end{proof}

One interesting question is that of understanding how much of (1) can apply to the general case. The answer here is as follows, with (2) providing counterexamples:

\begin{proposition}
We have a quotient map as follows, which is in general proper,
$$G/(G\cap C_N^{I+})\to X_{G,I}$$
where $C_N^{I+}\subset U_N^+$ is the subgroup defined by $C(C_N^{I+})=C(U_N^+)/<u\xi_I=\xi_I>$, with the relation $u\xi_I=\xi_I$ being interpreted as an equality of column vectors, over $C(U_N^+)$.
\end{proposition}

\begin{proof}
Observe first that $C_N^{I+}$ is indeed a quantum group. In fact, it is standard to exhibit an isomorphism $C_N^{+I}\simeq U_{N-1}^+$, by reasoning as in \cite{rau}. We must check that the defining relations for $C(G/(G\cap C_N^{I+}))$ hold for the standard generators $x_i\in C(X_{G,I})$. But if we denote by $\pi:C(G)\to C(G\cap C_N^{I+})$ the quotient map, we have, as desired:
$$(id\otimes\pi)\Delta x_i=(id\otimes\pi)\left(\frac{1}{\sqrt{|I|}}\sum_{j\in I}\sum_ku_{ik}\otimes u_{kj}\right)=\sum_ku_{ik}\otimes(\xi_I)_k=x_i\otimes1$$

Finally, for the group duals this quotient map is given by $\widehat{\Gamma}_I'\to\widehat{\Gamma}_I$, where $\Gamma_I'\subset\Gamma$ is the normal closure of $\Gamma_I$, and so this map can be indeed proper. See \cite{ban}.
\end{proof}

\section{Algebraic manifolds}

We discuss in what follows the axiomatization of the affine homogeneous spaces, as algebraic submanifolds of the free sphere $S^{N-1}_{\mathbb C,+}$. We use the following formalism:

\begin{definition}
A closed subset $X\subset S^{N-1}_{\mathbb C,+}$ is called algebraic when
$$C(X)=C(S^{N-1}_{\mathbb C,+})\Big/\Big<P_i(x_1,\ldots,x_N)=0,\forall i\in I\Big>$$
for a certain family of noncommutative $*$-polynomials $P_i\in\mathbb C<x_1,\ldots,x_N>$.
\end{definition}

There are many examples of such manifolds, as for instance all the compact matrix quantum groups. Indeed, assuming that we have a closed subgroup $G\subset U_N^+$, by rescaling the standard coordinates we obtain an embedding $G\subset U_N^+\subset S^{N^2-1}_{\mathbb C,+}$, and the following result, coming from \cite{wo2}, shows that we have indeed an algebraic manifold:

\begin{proposition}
Given a closed subgroup $G\subset U_N^+$, with the corresponding fundamental corepresentations denoted $u\to v$, we have the formula
$$C(G)=C(U_N^+)\Big/\Big(T\in Hom(u^{\otimes k},u^{\otimes l}),\forall k,l,\forall T\in Hom(v^{\otimes k},v^{\otimes l})\Big)$$
with $k,l=\ldots\circ\bullet\bullet\circ\bullet\ldots$ being colored integers, with the tensor power conventions $w^\circ=w,w^\bullet=\bar{w},w^{xy}=w^x\otimes w^y$, and with the notation $Hom(r,p)=\{T|Tr=pT\}$.
\end{proposition}

\begin{proof}
For any choice of two colored integers $k,l$ and of an intertwiner $T\in Hom(v^{\otimes k},v^{\otimes l})$, the formula $T\in Hom(u^{\otimes k},u^{\otimes l})$, with $u=(u_{ij})$ being the fundamental corepresentation of $C(U_N^+)$, corresponds to a collection of $N^{k+l}$ relations between the variables $u_{ij}$. By dividing now $C(U_N^+)$ by the ideal generated by all these relations, when $k,l$ and $T$ vary, we obtain a certain algebra $A$, which is the algebra on the right in the statement.

It is clear that we have a surjective morphism $A\to C(G)$, and by using Woronowicz's Tannakian results in \cite{wo2}, this surjective morphism follows  to be an isomorphism. For a short, recent proof of this fact, using basic Hopf algebra theory, see \cite{mal}.
\end{proof}

In relation now with the affine homogeneous spaces, we have:

\begin{proposition}
Any affine homogeneous space $X_{G,I}\subset S^{N-1}_{\mathbb C,+}$ is algebraic, with
$$\sum_{i_1\ldots i_k}\xi_{i_1\ldots i_k}x_{i_1}^{e_1}\ldots x_{i_k}^{e_k}=\frac{1}{\sqrt{|I|^k}}\sum_{b_1\ldots b_k\in I}\xi_{b_1\ldots b_k}\quad\forall k,\forall\xi\in Fix(u^{\otimes k})$$
as defining relations. Moreover, we can use vectors $\xi$ belonging to a basis of $Fix(u^{\otimes k})$.
\end{proposition}

\begin{proof}
This follows indeed from the various results in section 1, and more specifically from Proposition 1.6, by using the identifications made in Theorem 1.7.
\end{proof}

In order to reach to a more categorical description of $X_{G,I}$, the idea will be that of using Frobenius duality. We use colored indices, and we denote by $k\to\bar{k}$ the operation on the colored indices which consists in reversing the index, and switching all the colors. Also, we agree to identify the linear maps $T:(\mathbb C^N)^{\otimes k}\to(\mathbb C^N)^{\otimes l}$ with the corresponding rectangular matrices $T\in M_{N^l\times N^k}(\mathbb  C)$, written $T=(T_{i_1\ldots i_l,j_1\ldots j_k})$. With these conventions, the precise formulation of Frobenius duality that we will need is as follows:

\begin{proposition}
We have an isomorphism of complex vector spaces
$$T\in Hom(u^{\otimes k},u^{\otimes l})\ \longleftrightarrow\ \xi\in Fix(u^{\otimes l}\otimes u^{\otimes\bar{k}})$$
given by the formulae $T_{i_1\ldots i_l,j_1\ldots j_k}=\xi_{i_1\ldots i_lj_k\ldots j_1}$ and $\xi_{i_i\ldots i_lj_1\ldots j_k}=T_{i_1\ldots i_l,j_k\ldots j_1}$. 
\end{proposition}

\begin{proof}
This is a well-known result, which follows from the general theory in \cite{wo1}. To be more precise, given integers $K,L\in\mathbb N$, consider the following standard isomorphism, which in matrix notation makes $T=(T_{IJ})\in M_{L\times K}(\mathbb C)$ correspond to $\xi=(\xi_{IJ})$:
$$T\in\mathcal L(\mathbb C^{\otimes K},\mathbb C^{\otimes L})\ \longleftrightarrow\ \xi\in\mathbb C^{\otimes L+K}$$

Given now two arbitrary corepresentations $v\in M_K(C(G))$ and $w\in M_L(C(G))$, the abstract Frobenius duality result established in \cite{wo1} states that the above isomorphism restricts into an isomorphism of vector spaces, as follows:
$$T\in Hom(v,w)\ \longleftrightarrow\ \xi\in Fix(w\otimes\bar{v})$$

In our case, we can apply this result with $v=u^{\otimes k}$ and $w=u^{\otimes l}$. Since, according to our conventions, we have $\bar{v}=u^{\otimes\bar{k}}$, this gives the isomorphism in the statement.
\end{proof}

With the above result in hand, we can enhance the construction of $X_{G,I}$, as follows:

\begin{theorem}
Any affine homogeneous space $X_{G,I}$ is algebraic, with
$$\sum_{i_1\ldots i_l}\sum_{j_1\ldots j_k}T_{i_1\ldots i_l,j_1\ldots j_k}x_{i_1}^{e_1}\ldots x_{i_l}^{e_l}(x_{j_1}^{f_1}\ldots x_{j_k}^{f_k})^*=\frac{1}{\sqrt{|I|^{k+l}}}\sum_{b_1\ldots b_l\in I}\sum_{c_1\ldots c_k\in I}T_{b_1\ldots b_l,c_1\ldots c_k}$$
for any $k,l$, and any $T\in Hom(u^{\otimes k},u^{\otimes l})$, as defining relations.
\end{theorem}

\begin{proof}
We must prove that the relations in the statement are satisfied, over $X_{G,I}$. We know from Proposition 2.3 above, with $k\to l\bar{k}$, that the following relation holds:
$$\sum_{i_1\ldots i_l}\sum_{j_1\ldots j_k}\xi_{i_1\ldots i_lj_k\ldots j_1}x_{i_1}^{e_1}\ldots x_{i_l}^{e_l}x_{j_k}^{\bar{f}_k}\ldots x_{j_1}^{\bar{f}_1}=\frac{1}{\sqrt{|I|^{k+l}}}\sum_{b_1\ldots b_l\in I}\sum_{c_1\ldots c_k\in I}\xi_{b_1\ldots b_lc_k\ldots c_1}$$

In terms of the matrix $T_{i_1\ldots i_l,j_1\ldots j_k}=\xi_{i_1\ldots i_lj_k\ldots j_1}$ from Proposition 2.3, we obtain:
$$\sum_{i_1\ldots i_l}\sum_{j_1\ldots j_k}T_{i_1\ldots i_l,j_1\ldots j_k}x_{i_1}^{e_1}\ldots x_{i_l}^{e_l}x_{j_k}^{\bar{f}_k}\ldots x_{j_1}^{\bar{f}_1}=\frac{1}{\sqrt{|I|^{k+l}}}\sum_{b_1\ldots b_l\in I}\sum_{c_1\ldots c_k\in I}T_{b_1\ldots b_l,c_1\ldots c_k}$$

But this gives the formula in the statement, and we are done.
\end{proof}

\section{Tannakian duality}

In this section we state and prove our main result. The description of the affine homogeneous spaces found in Theorem 2.5 above suggests the following notion:

\begin{definition}
Given an algebraic submanifold $X\subset S^{N-1}_{\mathbb C,+}$ and a subset $I\subset\{1,\ldots,N\}$, we say that $X$ is $I$-affine when $C(X)$ is presented by relations of type
$$\sum_{i_1\ldots i_l}\sum_{j_1\ldots j_k}T_{i_1\ldots i_l,j_1\ldots j_k}x_{i_1}^{e_1}\ldots x_{i_l}^{e_l}(x_{j_1}^{f_1}\ldots x_{j_k}^{f_k})^*=\frac{1}{\sqrt{|I|^{k+l}}}\sum_{b_1\ldots b_l\in I}\sum_{c_1\ldots c_k\in I}T_{b_1\ldots b_l,c_1\ldots c_k}$$
with the operators $T$ belonging to certain linear spaces $F(k,l)\subset M_{N^l\times N^k}(\mathbb C)$, which altogether form a tensor category $F=(F(k,l))$.
\end{definition}

According to Theorem 2.5, any affine homogeneous space $X_{G,I}$ is an $I$-affine manifold, with the corresponding tensor category being the one associated to the quantum group $G\subset U_N^+$ which produces it, formed by the linear spaces $F(k,l)=Hom(u^{\otimes k},u^{\otimes l})$.

We will need some basic facts regarding the quantum affine actions. Following Definition 1.2, we say that a closed subgroup $G\subset U_N^+$ acts affinely on a closed subset $X\subset S^{N-1}_{\mathbb C,+}$ when the formula $\Phi(x_i)=\sum_ju_{ij}\otimes x_j$ defines a morphism of $C^*$-algebras.

We have the following standard result, from \cite{bme}, inspired from \cite{go1}, \cite{go2}:

\begin{proposition}
Given an algebraic manifold $X\subset S^{N-1}_{\mathbb C,+}$, the quantum group
$$G^+(X)=\max\left\{G\subset U_N^+\Big|G\curvearrowright X\right\}$$
exists, and is unique. We call it affine quantum isometry group of $X$.
\end{proposition}

\begin{proof}
In order to have a universal coaction, the relations defining $G^+(X)\subset U_N^+$ must be those making $x_i\to X_i=\sum_ju_{ij}\otimes x_j$ a morphism of algebras. Thus, in order to construct $G^+(X)$, we just have to clarify how the relations $P_\alpha(x_1,\ldots,x_N)=0$ defining $X$ are interpreted inside $C(U_N^+)$. So, pick one such polynomial, $P=P_\alpha$, and write it:
$$P(x_1,\ldots,x_N)=\sum_r\alpha_r\cdot x_{i_1^r}\ldots x_{i_{s(r)}^r}$$

Now if we formally replace each coordinate $x_i\in C(X)$ by the corresponding element $X_i=\sum_ju_{ij}\otimes x_j\in C(U_N^+)\otimes C(X)$, the following formula must hold:
$$P(X_1,\ldots,X_N)=\sum_r\alpha_r\sum_{j_1^r\ldots j_{s(r)}^r}u_{i_1^rj_1^r}\ldots u_{i_{s(r)}^rj_{s(r)}^r}\otimes x_{j_1^r}\ldots x_{j_{s(r)}^r}$$

Thus the relations $P(X_1,\ldots,X_N)=0$ correspond indeed to certain polynomial relations between the standard generators $u_{ij}\in C(U_N^+)$, and this gives the result. See \cite{bme}.
\end{proof} 

Now by getting back to our questions, let us study the quantum isometry groups of the manifolds $X\subset S^{N-1}_{\mathbb C,+}$ which are $I$-affine. We have here the following result:

\begin{proposition}
For an $I$-affine manifold $X\subset S^{N-1}_{\mathbb C,+}$ we have 
$$G\subset G^+(X)$$
where $G\subset U_N^+$ is the Tannakian dual of the associated tensor category $F$.
\end{proposition}

\begin{proof}
We recall from the proof of Proposition 3.2 above that the relations defining $G^+(X)$ are those expressing the vanishing of the following quantities:
$$P(X_1,\ldots,X_N)=\sum_r\alpha_r\sum_{j_1^r\ldots j_{s(r)}^r}u_{i_1^rj_1^r}\ldots u_{i_{s(r)}^rj_{s(r)}^r}\otimes x_{j_1^r}\ldots x_{j_{s(r)}^r}$$

In the case of an $I$-affine manifold, the defining relations are those from Definition 3.1 above, with the corresponding polynomials $P$ being indexed by the elements of $F$. But the vanishing of the associated relations $P(X_1,\ldots,X_N)=0$ corresponds precisely to the Tannakian relations defining $G\subset U_N^+$, and so we obtain $G\subset G^+(X)$, as claimed.
\end{proof}

We have now all the needed ingredients, and we can prove:

\begin{theorem}
Assuming that an algebraic manifold $X\subset S^{N-1}_{\mathbb C,+}$ is $I$-affine, with associated tensor category $F$, the following happen:
\begin{enumerate}
\item We have an inclusion $G\subset G^+(X)$, where $G$ is the Tannakian dual of $F$.

\item $X$ is an affine homogeneous space, $X=X_{G,I}$, over this quantum group $G$.
\end{enumerate}
\end{theorem}

\begin{proof}
In the context of Definition 3.1, the tensor category $F$ there gives rise, by the Tannakian result in Proposition 2.2, to a quantum group $G\subset U_N^+$. What is left now is to construct the affine space morphisms $\alpha,\Phi$, and the proof here goes as follows:

(1) Construction of $\alpha$. We want to construct a morphism, as follows:
$$\alpha:C(X)\to C(G)\quad:\quad x_i\to X_i=\frac{1}{\sqrt{|I|}}\sum_{j\in I}u_{ij}$$

In view of Definition 3.1, we must therefore prove that we have:
$$\sum_{i_1\ldots i_l}\sum_{j_1\ldots j_k}T_{i_1\ldots i_l,j_1\ldots j_k}X_{i_1}^{e_1}\ldots X_{i_l}^{e_l}(X_{j_1}^{f_1}\ldots X_{j_k}^{f_k})^*=\frac{1}{\sqrt{|I|^{k+l}}}\sum_{b_1\ldots b_l\in I}\sum_{c_1\ldots c_k\in I}T_{b_1\ldots b_l,c_1\ldots c_k}$$

By replacing the variables $X_i$ by their above values, we want to prove that:
$$\sum_{i_1\ldots i_l}\sum_{j_1\ldots j_k}\sum_{r_1\ldots r_l\in I}\sum_{s_1\ldots s_k\in I}T_{i_1\ldots i_l,j_1\ldots j_k}u_{i_1r_1}^{e_1}\ldots u_{i_lr_l}^{e_l}(u_{j_1s_1}^{f_1}\ldots u_{j_ks_k}^{f_k})^*=\sum_{b_1\ldots b_l\in I}\sum_{c_1\ldots c_k\in I}T_{b_1\ldots b_l,c_1\ldots c_k}$$

Now observe that from the relation $T\in Hom(u^{\otimes k},u^{\otimes l})$ we obtain:
$$\sum_{i_1\ldots i_l}\sum_{j_1\ldots j_k}T_{i_1\ldots i_l,j_1\ldots j_k}u_{i_1r_1}^{e_1}\ldots u_{i_lr_l}^{e_l}(u_{j_1s_1}^{f_1}\ldots u_{j_ks_k}^{f_k})^*=T_{r_1\ldots r_l,s_1\ldots s_k}$$

Thus, by summing over indices $r_i\in I$ and $s_i\in I$, we obtain the desired formula.

(2) Construction of $\Phi$. We want to construct a morphism, as follows:
$$\Phi:C(X)\to C(G)\otimes C(X)\quad:\quad x_i\to X_i=\sum_ju_{ij}\otimes x_j$$

But this is precisely the coaction map constructed in Proposition 3.3 above.

(3) Proof of the ergodicity. If we go back here to Proposition 1.4, we see that the ergodicty condition is equivalent to a number of Tannakian conditions, which are automatic in our case. Thus, the ergodicity condition is automatic, and we are done.
\end{proof}

\section{Further results}

The Tannakian result obtained in section 3 above, based on the notion of $I$-affine manifold from Definition 3.1, remains quite theoretical. The problem is that Definition 3.1 still makes reference to a tensor category, and so the abstract characterization of the affine homogeneous spaces that we obtain in this way is not totally intrinsic. 

We believe that some deeper results should hold as well. To be more precise, the work on noncommutative spheres in \cite{bme} suggests that the relevant category $F$ should appear in a more direct way from $X$. In analogy with Definition 3.1, let us formulate:

\begin{definition}
Given a submanifold $X\subset S^{N-1}_{\mathbb C,+}$ and a subset $I\subset\{1,\ldots,N\}$, we let $F_{X,I}(k,l)\subset M_{N^l\times N^k}(\mathbb C)$ be the linear space of linear maps $T$ such that
$$\sum_{i_1\ldots i_l}\sum_{j_1\ldots j_k}T_{i_1\ldots i_l,j_1\ldots j_k}x_{i_1}^{e_1}\ldots x_{i_l}^{e_l}(x_{j_1}^{f_1}\ldots x_{j_k}^{f_k})^*=\frac{1}{\sqrt{|I|^{k+l}}}\sum_{b_1\ldots b_l\in I}\sum_{c_1\ldots c_k\in I}T_{b_1\ldots b_l,c_1\ldots c_k}$$
holds over $X$. We say that $X$ is $I$-saturated when $F_{X,I}=(F_{X,l}(k,l))$ is a tensor category, and the collection of the above relations presents the algebra $C(X)$.
\end{definition}

Observe that any $I$-saturated manifold is automatically $I$-affine. The point is that the results in \cite{bme} seem to suggest that the converse of this fact should hold, in the sense that any $I$-affine manifold should be automatically $I$-saturated. Such a result would of course substantially improve Theorem 3.4 above, and make it ready for applications.

We do not have a proof of this fact, but we would like to present now a few preliminary observations on this subject. First of all, we have the following result:

\begin{proposition}
The linear spaces $F_{X,I}(k,l)\subset M_{N^l\times N^k}(\mathbb C)$ constructed above have the following properties:
\begin{enumerate}
\item They contain the units.

\item They are stable by conjugation.

\item They satisfy the Frobenius duality condition.
\end{enumerate}
\end{proposition}

\begin{proof}
All these assertions are elementary, as follows:

(1) Consider indeed the unit map. The associated relation is:
$$\sum_{i_1\ldots i_k}x_{i_1}^{e_1}\ldots x_{i_k}^{e_k}(x_{i_1}^{e_1}\ldots x_{i_k}^{e_k})^*=1$$

But this relation holds indeed, due to the defining relations for $S^{N-1}_{\mathbb C,+}$.

(2) We have indeed the following sequence of equivalences:
\begin{eqnarray*}
&&T^*\in F_{X,I}(l,k)\\
&\iff&\sum_{i_1\ldots i_l}\sum_{j_1\ldots j_k}T^*_{j_1\ldots j_k,i_1\ldots i_l}x_{j_1}^{f_1}\ldots x_{j_k}^{f_k}(x_{i_1}^{e_1}\ldots x_{i_l}^{e_l})^*=\frac{1}{\sqrt{|I|^{k+l}}}\sum_{b_1\ldots i_l\in I}\sum_{c_1\ldots c_k\in I}T^*_{c_1\ldots c_k,b_1\ldots b_l}\\
&\iff&\sum_{i_1\ldots i_l}\sum_{j_1\ldots j_k}T_{i_1\ldots i_l,j_1\ldots j_k}x_{i_1}^{e_1}\ldots x_{i_l}^{e_l}(x_{j_1}^{f_1}\ldots x_{j_k}^{f_k})^*=\frac{1}{\sqrt{|I|^{k+l}}}\sum_{b_1\ldots b_l\in I}\sum_{c_1\ldots c_k\in I}T_{b_1\ldots b_l,c_1\ldots c_k}\\
&\iff&T\in F_{X,I}(k,l)
\end{eqnarray*}

(3) We have indeed a correspondence $T\in F_{X,I}(k,l)\ \longleftrightarrow\ \xi\in F_{X,I}(\emptyset,l\bar{k})$, given by the usual formulae for the Frobenius isomorphism, from Proposition 2.4.
\end{proof}

Based on the above result, we can now formulate our observations, as follows:

\begin{theorem}
Given a closed subgroup $G\subset U_N^+$, and an index set $I\subset\{1,\ldots,N\}$, consider the corresponding affine homogeneous space $X_{G,I}\subset S^{N-1}_{\mathbb C,+}$.
\begin{enumerate}
\item $X_{G,I}$ is $I$-saturated precisely when the collection of spaces $F_{X,I}=(F_{X,I}(k,l))$ is stable under compositions, and under tensor products.

\item We have $F_{X,I}=F$ precisely when $\sum_{j_1\ldots j_l\in I}(\sum_{i_1\ldots i_l}\xi_{i_1\ldots i_l}u_{i_1j_1}^{e_1}\ldots u_{i_lj_l}^{e_l}-\xi_{j_1\ldots j_l})=0$ implies $\sum_{i_1\ldots i_l}\xi_{i_1\ldots i_l}u_{i_1j_1}^{e_1}\ldots u_{i_lj_l}^{e_l}-\xi_{j_1\ldots j_l}=0$, for any $j_1,\ldots,j_l$.
\end{enumerate}
\end{theorem}

\begin{proof}
We use the fact, coming from Theorem 2.5, that with $F(k,l)=Hom(u^{\otimes k},u^{\otimes l})$, we have inclusions of vector spaces $F(k,l)\subset F_{X,I}(k,l)$. Moreover, once again by Theorem 2.5, the relations coming from the elements of the category formed by the spaces $F(k,l)$ present $X_{G,I}$. Thus, the relations coming from the elements of $F_{X,I}$ present $X_{G,I}$ as well.

With this observation in hand, our assertions follow from Proposition 4.2:

(1) According to Proposition 4.2 (1) and (2) the unit and conjugation axioms are satisfied, so the spaces $F_{X,I}(k,l)$ form a tensor category precisely when the remaining axioms, namely the composition and the tensor product one, are satisfied. Now by assuming that these two axioms are satisfied, $X$ follows to be $I$-saturated, by the above observation.

(2) Since we already have inclusions in one sense, the equality $F_{X,I}=F$ from the statement means that we must have inclusions in the other sense, as follows:
$$F_{X,I}(k,l)\subset F(k,l)$$

By using now Proposition 4.2 (3), it is enough to discuss the case $k=0$. And here, assuming that we have $\xi\in F_{X,L}(0,l)$, the following condition must be satisfied:
$$\sum_{i_1\ldots i_l}\xi_{i_1\ldots i_l}x_{i_1}^{e_1}\ldots x_{i_l}^{e_l}=\sum_{j_1\ldots j_l\in I}\xi_{j_1\ldots j_l}$$

By applying now the morphism $\alpha:C(X_{G,I})\to C(G)$, we deduce that we have:
$$\sum_{i_1\ldots i_l}\xi_{i_1\ldots i_l}\sum_{j_1\ldots j_l\in I}u_{i_1j_1}^{e_1}\ldots u_{i_lj_l}^{e_l}=\sum_{j_1\ldots j_l\in I}\xi_{j_1\ldots j_l}$$

Now recall that $F(0,l)=Fix(u^{\otimes l})$ consists of the vectors $\xi$ satisfying:
$$\sum_{i_1\ldots i_l}\xi_{i_1\ldots i_l}u_{i_1j_1}^{e_1}\ldots u_{i_lj_l}^{e_l}=\xi_{j_1\ldots j_l},\forall j_1,\ldots,j_l$$

We are therefore led to the conclusion in the statement.
\end{proof}

It is quite unclear on how to advance on these questions, and a more advanced algebraic trick, in the spirit of those used in \cite{bme}, seems to be needed. Nor is it clear on how to explicitely ``capture'' the relevant subgroup $G\subset G^+(X)$, in terms of our given manifold $X=X_{G,I}$, in a direct, geometric way. Summarizing, further improving Theorem 3.4 above is an interesting question, that we would like to raise here.


\begin{thebibliography}{99}

\bibitem{ban}T. Banica, Weingarten integration over noncommutative homogeneous spaces, {\em Ann. Math. Blaise Pascal} {\bf 24} (2017), 195--224.

\bibitem{bco}T. Banica and B. Collins, Integration over compact quantum groups, {\em Publ. Res. Inst. Math. Sci.} {\bf 43} (2007), 277--302.

\bibitem{bgo}T. Banica and D. Goswami, Quantum isometries and noncommutative spheres, {\em Comm. Math. Phys.} {\bf 298} (2010), 343--356.

\bibitem{bme}T. Banica and S. M\'esz\'aros, Uniqueness results for noncommutative spheres and projective spaces, {\em Illinois J. Math.} {\bf 59} (2015), 219--233.

\bibitem{cma}B. Collins and S. Matsumoto, Weingarten calculus via orthogonality relations: new applications, preprint 2016.

\bibitem{csn}B. Collins and P. \'Sniady, Integration with respect to the Haar measure on the unitary, orthogonal and symplectic group, {\em Comm. Math. Phys.} {\bf 264} (2006), 773--795.

\bibitem{go1}D. Goswami, Quantum group of isometries in classical and  noncommutative geometry, {\em Comm. Math. Phys.} {\bf 285} (2009), 141--160.

\bibitem{go2}D. Goswami, Existence and examples of quantum isometry groups for a class of compact metric spaces, {\em Adv. Math.} {\bf 280} (2015), 340--359.

\bibitem{mal}S. Malacarne, Woronowicz's Tannaka-Krein duality and free orthogonal quantum groups, preprint 2016.

\bibitem{rau}S. Raum, Isomorphisms and fusion rules of orthogonal free quantum groups and their complexifications, {\em Proc. Amer. Math. Soc.} {\bf 140} (2012), 3207--3218.

\bibitem{wan}S. Wang, Free products of compact quantum groups, {\em Comm. Math. Phys.} {\bf 167} (1995), 671--692.

\bibitem{wei}D. Weingarten, Asymptotic behavior of group integrals in the limit of infinite rank, {\em J. Math. Phys.} {\bf 19} (1978), 999--1001.

\bibitem{wo1}S.L. Woronowicz, Compact matrix pseudogroups, {\em Comm. Math. Phys.} {\bf 111} (1987), 613--665.

\bibitem{wo2}S.L. Woronowicz, Tannaka-Krein duality for compact matrix pseudogroups. Twisted SU(N) groups, {\em Invent. Math.} {\bf 93} (1988), 35--76.

\end{thebibliography}
\end{document}